\documentclass[12pt]{amsart}
\usepackage{latexsym}
\usepackage{amssymb}
\usepackage{hyperref}
\usepackage{nameref}

\newtheorem{Thm}{Theorem}[section] 
\newtheorem{Prop}[Thm]{Proposition} 
\newtheorem{Lem}[Thm]{Lemma} 
 
\theoremstyle{definition}
\newtheorem{Def}[Thm]{Definition}
\newtheorem{Alg}[Thm]{Algorithm}
\numberwithin{equation}{section} 

\newcommand{\Eqv}{\mathrm{Eqv}}
\newcommand{\Min}{\mathrm{min}}
\newcommand{\Eq}{\mathrm{Eq}}
\newcommand{\Crr}{\mathrm{Crr}}
\newcommand{\Crg}{\mathrm{Crg}}
\newcommand{\Free}{\mathbf{F}_{\vv{V}}}

\newcommand{\RS}{\mathrm{Crg}}
\newcommand{\vv}[1]{\mathcal {#1}} 
\DeclareMathAlphabet\mathbfsl {T1}{cmr}{bx}{it}

\newcommand{\NN}{{\mathbb{N}}}

\title[Compatible reflexive relations]{Mal'cev conditions corresponding to identities for compatible reflexive relations}
\author{Stefano Fioravanti}

\address{Stefano Fioravanti,
	Institut f\"ur Algebra,
	Johannes Kepler Universit\"at Linz,
	4040 Linz,
	Austria}
\email{\tt stefano.fioravanti66@gmail.com}

\urladdr{http://www.jku.at/algebra}
\thanks{Supported by the Austrian Science Fund (FWF):P29931.}
\date{\today}
\begin{document}
	
	\begin{abstract}
		
		We investigate Mal'cev conditions described by equations whose variables runs over the set of all compatible reflexive relations.
		Let $p \leq q$ be an equation for the language $\{\wedge, \circ,+\}$. We give a characterization of the class of all varieties which satisfy $p \leq q$ over the set of all compatible reflexive relations.
		The aim is to find an analogon of the Pixley-Wille algorithm \cite{Pix.LMC}, \cite{Wil.K} for conditions expressed by equations over the set of all compatible reflexive relations, and to characterize when an equation $p \leq q$ expresses the same property when considered over the congruence lattices or over the sets of all compatible reflexive relations of algebras in a variety.
		
	\end{abstract}

\maketitle

\section{Introduction}

Starting from Mal'cev's description of congruence permutability as in \cite{Mal.OTGT}, the problem of characterizing properties of classes of varieties as Mal'cev conditions has led to several results. In \cite{Pix.DAPO} A. Pixley found a strong Mal'cev condition defining the class of varieties with distributive and permuting congruences. In \cite{Jon.AWCL} B. J\'{o}nsson shows a Mal'cev condition characterizing congruence distributivity, in \cite{Day.ACOM} A. Day shows a Mal'cev condition characterizing the class of varieties with modular congruence lattices.

These results are examples of a more general theorem obtained independently by Pixley \cite{Pix.LMC} and R. Wille \cite{Wil.K} that can be considered as a foundational result in the field. They proved that if $p \leq q$ is a lattice identity, then the class of varieties whose congruence lattices satisfy $p \leq q$ is the intersection of countably many Mal'cev classes. \cite{Pix.LMC} and \cite{Wil.K} include an algorithm to generate Mal'cev conditions associated with congruence identities.

These researches have also led to the problem of studying equations where the variables run not only over the congruence lattices but in possibly strictly larger sets as the lattices of all tolerances or of all compatible reflexive relations. Results about this problem can be found in \cite{CHL.OMCF}, \cite{GM.QLOV}, \cite{Gum.GMIC}, \cite{KK.TSOC}, \cite{Lip.RII3}, \cite{Lip.UOAR}, \cite{Wer.AMCF}.

Let $R$ and $T$ be two binary relations on a set $A$. Then we define the $k$-\emph{fold relational product}, in symbols  $\circ^{(k)}$, 
of $R$ and $T$ by:

\begin{align*}	
R \circ^{(k)} T =& \{(a,b)  \in A^2 \mid \exists  c_0,\dots c_k \in A \text{ s.t. } c_0 = a, c_k = b,\\&(c_{i}, c_{i+1}) \in R \text{ for $i$ even}, (c_{i}, c_{i+1}) \in T \text{ for $i$ odd}\}
\end{align*}
We denote by $\circ$ the $2$-fold relational product of  two relations and we denote by $R+T = \bigcup_{i\in \NN}R\circ^{(i)}T$.

The aim of this paper is to study the relationship between congruence equations and compatible reflexive relation equations. We denote by \index{$\Crr(\mathbf{A})$}$\Crr(\mathbf{A})$ the set of all compatible reflexive relations of $\mathbf{A}$. Let $\mathbf{A}$ be an algebra and let $p$ and $q$ be terms for the language $\{\wedge, +, \circ\}$, where $\wedge$ is the set-theoretical intersection. Then we say that $\mathbf{A}$ \emph{satisfies the compatible reflexive relation equality (inequality)} $p \approx q$ $(p \leq q)$ if and only if $\Crr(\mathbf{A})$ satisfies $p \approx q$ $(p \leq q)$. Let $\vv{V}$ be a variety. Then we say that $\vv{V}$ satisfies $p \approx q$ $(p \leq q)$ if and only if every algebra in $\vv{V}$ satisfies $p \approx q$ $(p \leq q)$. These equations are of interest since Mal'cev conditions are often express in term of this type of equations.

 In Section \ref{Alg} we study the connection between Mal'cev conditions and compatible reflexive relation inequalities. We provide a self-contained proof of an algorithm, claimed by Tschanz \cite[pages 273-274]{Tsc.MCET} without a proof and similar to the one of Pixley-Wille \cite{Pix.LMC}, which allows us to compute a Mal’cev condition described by the property of satisfying a particular compatible reflexive relation equation. 

\begin{Thm}
	\label{sunewPWThm}
	Let $p$ and $q$ be two $n$-ary terms for the language $\{\wedge, \circ, +$ $\}$ with $p$ $+$-free. Then the class of varieties such that $p \leq q$ holds in $\Crr(A)$ for all $A \in \vv{V}$ is a Mal’cev class. Furthermore, if  $q$ is $+$-free then the class of varieties such that $p \leq q$ holds in $\Crr(A)$ for all $A \in \vv{V}$ is a strong Mal’cev class.
\end{Thm}

Our contribute will be the proof of this theorem that provides several examples of descriptions of Mal'cev conditions.

In Section \ref{SecLst} we give a characterization of an interesting set of equations which describe the same Mal'cev condition when considered with variables that run over the set of all compatible reflexive relations or over the congruence lattices. 

In \cite{Lip.FCIT} it has been shown that, under a weak assumption on the term $p$, a variety satisfies the congruence identity $p(\alpha_1,\dots, \alpha_n) \leq q(\alpha_1,\dots,\alpha_n)$ if and only if it satisfies the tolerance identity $p(T_1,\dots,T_n$ $) \leq q(T_1,\dots,T_n)$, provided we restrict the set of the variables to run over the representable tolerances, i.e. tolerances of the form $R \circ R^{-1}$ where $R$ is a compatible reflexive relation. 

Our main result, Theorem \ref{Theq}, provides a sufficient criterion to check whether an inequality $p \leq q$ describes the same Mal'cev condition considered as compatible reflexive relation inequality and as congruence inequality. The hypothesis of this theorem is quite restrictive for $p$ and $q$ but, in Section \ref{SecLst}, we provide two examples showing that is quite difficult to extend this hypothesis.

\section{Notation}
\label{Notations}
Let $\mathbf{A}$ be an algebra and let $X \subseteq A^2$. We denote by $\Crg_{\mathbf{A}}(X)$ the compatible reflexive relation generated by the set of pairs $X$. Let $R$ and $L$ be two binary relations on $A$. Furthermore, we denote by $\Eqv(R)$ the equivalence relation generated by a relation $R$.

Let $p$ be a term for the language $\{\wedge,+,\circ\}$. Let $k \in \NN$. We denote by $p^{(k)}$ the $\{\wedge,\circ\}$-term obtained from $p$ substituting any occurrence of $+$ with the $k$-fold relational product $\circ^{(k)}$.

We denote by \index{$[n]$}$[n]$ the set $\{i \in \NN\mid 1 \leq i \leq n\}$ and by \index{$[n]_0$}$[n]_0$ the set $[n] \cup \{0\}$.

\section{Labelled graphs and regular terms}
\label{LabelledG}
In order to show the main results we recall the definition of a regular term introduced in \cite{Lip.FCIT}.
Let $p$ be a term for the language $\{\circ,\wedge\}$. We define inductively the set \index{$L_p$}$L_p$ and \index{$R_p$}$R_p$ of \emph{variables on the left} and \emph{right side of} $p$ respectively.
\begin{enumerate}
	\item[(i)] If $p = X_i$ is a variable, then $L_p = \{X_i\}$ and $R_p = \{X_i\}$;
	\item[(ii)] if $p = q \circ r$, then $L_p = L_q$ and $R_p = R_r$;
	\item[(iii)] if $p = q \wedge r$, then $L_p = L_q \cup L_r$ and $R_p = R_q \cup R_r$.
\end{enumerate}

\begin{Def}
	\label{DefRterms}
	The class of \emph{regular terms} is the smallest class of $\{\circ, \wedge\}$-terms that:
	\begin{enumerate}
		\item [(1)] contains all the variables;
		\item [(2)] contains $p = q \circ r$ whenever $q$ and $r$ are regular and $R_q \cap L_r = \emptyset$;
		\item [(3)] contains $p = q \wedge r$ whenever $q$ and $r$ are regular, $L_q \cap L_r = \emptyset$, and $R_q \cap R_r = \emptyset$.
	\end{enumerate}
\end{Def}

Hence examples of regular terms are $X \wedge (Y \circ Z)$ and $X \wedge (Y \circ X \circ Y)$. Instead, $X \circ X$ and $X \wedge(Y \circ Z \circ X)$ are not regular.
We can see that the majority of the known Mal'cev conditions given by congruence equations can be characterized by a congruence equation of regular terms.

The definition of a regular term can be better understood by means of the notion of the labelled graph associated with a term as in \cite{Cze.ACFC} \cite{Cze.AMTC} \cite{Cze.OPOR} \cite{CD.HSWW} \cite{KK.TSOC}. We show how to build the labelled graph associated with a $\{\circ,\wedge\}$-term $p$.

\begin{Def}
	Let $S$ be a set of labels. Then a \emph{labelled graph} is a directed graph $(V,E)$ with a labelling function $l: E \rightarrow S$.
\end{Def}

We denote by $(v_1,v_2)$ an edge connecting the vertex $v_1$ to $v_2$. Following \cite{KK.TSOC}, let $p$ be a $\{\wedge,\circ\}$-term. We start with the graph $\mathbf{G}_1(p)$ having an edge $(y_1,y_2)$, labelled with $p$, connecting two vertices $y_1$ and $y_2$.

Starting from $\mathbf{G}_1(p)$ we build a finite sequence of graphs $\mathbf{G}_1(p), \dots, $ $\mathbf{G}_l(p) = \mathbf{G}(p)$ such that for all $i \in [l]$ we select an edge $(y_j,y_k)$, connecting vertices $y_k$ and $y_j$, from $\mathbf{G}_i(p)$ labelled with a term $w$ that is not a variable. Then we have two cases:

if $w = u \wedge v$, then $\mathbf{G}_{i+1}(p)$ is obtained from $\mathbf{G}_i(p)$ by replacing the edge $(y_j,y_k)$ labelled $w$ with two edges $(y_j,y_k)$ labelled $u$ and $v$ respectively, and both connecting the same vertices.

If $w = u \circ v$, then $\mathbf{G}_{i+1}(p)$ is obtained from $\mathbf{G}_i(p)$ by introducing a new vertex $y_t$ and replacing the edge labelled $w$ with two edges $(y_j,y_t)$ and $(y_t,y_k)$, labelled $u$ and $v$ respectively, and connecting the same vertices in serial through the new vertex. The sequence ends when all the edges are labelled with variables. 

We can observe that a term is regular if in its graph, all the edges adjacent to any given vertex are labelled differently.

We denote by $Y = \{y_1 ,\dots, y_m\}$ the set of vertices and we denote by capital letters the variables of $p$ labelling the graph $\mathbf{G}(p)$. 
The main reason to introduce $\mathbf{G}(p)$ is stated in \cite[Proposition 3.1]{CD.HSWW} and in   Claim $4.8$ of \cite{KK.TSOC}. The latter can be generalized to tolerances as in \cite[Proposition $2.1$]{Lip.FCIT} and also to relations in general.

\begin{Prop}
	\label{PropKK}
	Let $\mathbf{A}$ be an algebra, let $R_i \subseteq A \times A$, for $1 \leq i \leq n$, and let $p$ be a $\{\circ ,\wedge\}$-term. Then:
	
	\begin{enumerate}
		\item[(1)]  Let $Y \rightarrow A$: $y_s \mapsto a_s$ be an assignment such that for all edges $(y_i, y_j)$  with label $X_k$ of $\mathbf{G}(p)$, we have $(a_i,a_j) \in R_k$. Then $(a_1, a_2) \in p(R_1,\dots,R_n)$.
		
		\item[(2)] Conversely, given any $(a_1, a_2) \in p(R_1,\dots,R_n)$, there is an assignment $Y \rightarrow A$: $y_s \mapsto a_s$ extending $y_1 \mapsto a_1$, $y_2 \mapsto a_2$ such that $(a_i,a_j) \in R_k$ whenever $(y_i, y_j)$ is an edge labelled with $X_k$ of $\mathbf{G}(p)$, where $(y_1,y_2)$ is the only edge of the graph $\mathbf{G}_1(p)$.
	\end{enumerate}
\end{Prop}

 \section{The Algorithm}
\label{Alg}
In this section we show the Algorithm in \cite{Pix.LMC} and \cite{Wil.K} and we then modify the last part of it in order to obtain a Mal'cev condition describing the class of varieties that satisfy the compatible reflexive relation equation $p \leq q$.

	Let $p \leq q$ be an equation for the language $\{\wedge, \circ\}$. Let $\mathbf{G}(p)$ and $\mathbf{G}(q)$ be obtained from $p$ and $q$ (Definition in Section \ref{Notations}) with the procedure in Section \ref{LabelledG}. Then we define the sets of pairs:
\begin{align}
	\label{eqp}
	T_s(p) &:= \{(x_i,x_j) \mid (y_i,y_j) \text{ is an edge of } \mathbf{G}(p) \text{ with label } X_s \};
	\\Tt_s(q) &:=\{(t_i,t_j) \mid (y_i,y_j) \text{ is an edge of } \mathbf{G}(q) \text{ with label } X_s\},\nonumber
\end{align}
where the elements $\{t_1,\dots,t_l\}$ of the pairs in $Tt_s(q)$ are $m$-ary terms with $t_1 = \pi_1^m$ and $t_2 = \pi_2^m$ in the variables $\{x_1,\dots,x_m\}$, $m$ is the number of vertices  in $\mathbf{G}(p)$, and where $\pi_i^m$ is the $m$-ary term $x_i$. 

\begin{Alg}
	\label{PixWilAlg}
	\theoremstyle{definition}
	Let $p \leq q$ be an equation for the language $\{\wedge, \circ\}$. Let $\mathbf{G}(p)$ and $\mathbf{G}(q)$ be obtained from $p$ and $q$ with the procedure in Section \ref{LabelledG}. Let us consider $T_s(p) $ and $Tt_s(q) $ as in \eqref{eqp}.
	
	Next we define $\Eq(p \leq q)$ as the set of all equations of the form:
	\begin{equation*}
		t_i(x_{i_1},\dots,x_{i_m}) \approx t_j(x_{i_1},\dots,x_{i_m})
	\end{equation*}
	such that $(t_i,t_j) \in Tt_s(q)$ and the vector of indices $(i_1,\dots,i_m) \subseteq \NN^n$ satisfies $i_d = \Min(i \mid (x_{i}, x_d) \in \Eqv(T_s(p))) $ for all $d \in [m]$. This means that the variables that are in the equivalence relation generated by the pairs in $T_s(p)$ are collapsed.

\end{Alg}

	We modify the last part of the Algorithm in \cite{Pix.LMC}, \cite{Wil.K} in order to obtain a set of equations $\Eq^{R}(p \leq q)$ that characterizes the Mal'cev condition describing the class of varieties which satisfy $p \leq q$ over the set of all compatible reflexive relations.
	
	\begin{Alg}
		\label{PW-Alg}
		\theoremstyle{definition}
		Let $p \leq q$ be an equation for the language $\{\wedge, \circ\}$. Let $\mathbf{G}(p)$ and $\mathbf{G}(q)$ be obtained from $p$ and $q$ with the procedure in Section \ref{LabelledG}. Let us consider $T_s(p) $ and $Tt_s(q) $ as in \eqref{eqp}.
		
	Let $c(s)$ be the cardinality of $T_s(p)$ and let us set an order of these pairs $\{(y_{i_t},y_{j_t})\}_{1 \leq t \leq c(s)} = T_s(p)$. We define $\Eq^{R}(p \leq q)$ as the set of all equations of the form:
	\begin{align*}
		t_{(i,j,s)}(x_1,\dots,x_m,y_{i_1},\dots,y_{i_{c(s)}}) &\approx t_i(x_1,\dots,x_m)
		\\ t_{(i,j,s)}(x_1,\dots,x_m,y_{j_1},\dots,y_{j_{c(s)}}) &\approx t_j(x_1,\dots,x_m),
	\end{align*}	
	such that $(t_i,t_j) \in Tt_s(q)$, $t_{(i,j,s)}$ is an $(m + c(s))$-ary term. 
\end{Alg}

Note that the $c(s)!$ ways to generate the equations given by $(t_i,t_j) \in Tt_s(q)$ and $T_s(p)$ in the Algorithm \ref{PW-Alg} give different terms $t_{(i,j,s)}$s but equivalent conditions, up to reordering of the variables. 

\begin{Lem}
	\label{LemmClas2}
	Let $p(X_1,\dots,X_n)$ be an $m$-ary regular $\{\circ,\wedge\}$-term. Let $\{T_i(p)\}_{1 \leq i \leq n}$ be obtained with the Algorithm \ref{PW-Alg}. Let $V = \{y_1,\dots,y_m\}$ be the set of vertices of $\mathbf{G}(p)$. Then for all $i \in [n] $ the equivalence relation on $V$ generated by the pairs in $T_i(p)$ is composed by classes of cardinality at most $2$.
\end{Lem}

\begin{proof}
	The proof follows from the definition of regular term since in a graph of a regular term there cannot be two adjacent edges labelled with the same variable.
\end{proof}

The next lemma is the main ingredient to prove Theorem \ref{Theq}. Indeed, this allows us to understand some properties fulfilled by Mal'cev conditions described by congruence inequalities with at most one occurrence of $\circ$ in the right hand side of the equation.

Let $\mathbfsl{a}$ be a vector and let $R$ be an irreflexive and antisymmetric relation over the set of all components of $\mathbfsl{a}$ which generates an equivalence relation with classes of at most $2$ elements. This means that $\Eqv(R)$ is associated with a graph whose edges are disjoint and undirected. Let us denote by \index{$\mathbfsl{a}^{(R,l)}$}$\mathbfsl{a}^{(R,l)}$ and \index{$\mathbfsl{a}^{(R,r)}$}$\mathbfsl{a}^{(R,r)}$ the vectors such that if $((\mathbfsl{a})_i,(\mathbfsl{a})_j) \in R$, then $(\mathbfsl{a}^{(R,l)})_i = (\mathbfsl{a}^{(R,l)})_j = (\mathbfsl{a})_i$, $(\mathbfsl{a}^{(R,r)})_i = (\mathbfsl{a}^{(R,r)})_j = (\mathbfsl{a})_j$.  If  $(\mathbfsl{a})_i$ does not occur in pairs of $R$, then $(\mathbfsl{a}^{(R,l)})_i = (\mathbfsl{a}^{(R,r)})_i = (\mathbfsl{a})_i$. 
\begin{Lem}
	\label{Lemeqlr}
	Let $p,q$ be terms for the language $\{\circ,\wedge\}$, with $p$ regular. Let $(t_i,t_j) \in Tt_k(q)$ with $k \in \NN$. Then the following equations are logical consequences of $\Eq(p \leq q)$:
	\begin{equation*}
	t_i(\mathbfsl{x}^{(T_k(p),l)}) \approx t_j(\mathbfsl{x}^{(T_k(p),l)})
	\end{equation*}
	\begin{equation*}
	t_i(\mathbfsl{x}^{(T_k(p),r)}) \approx t_j(\mathbfsl{x}^{(T_k(p),r)}).
	\end{equation*}
\end{Lem}

The next lemma is what allows us to characterize conditions described by compatible reflexive relation equations and is an equivalent statement  for compatible reflexive relations of what is also called Mal'cev argument (\cite[Lemma $12.1$]{BS.ACIU}).

\begin{Lem}
	\label{Lem1-4}
	Let $\vv{V}$ be a variety of type $\mathcal{F}$ and let $\Free(\overline{X})$ be the free algebra generated by $\overline{X} = \{\overline{x}_1,\dots,\overline{x}_n\}$ in $\vv{V}$. Let $N \subseteq \overline{X}^2\backslash\Delta_{\overline{X}}$ where $\Delta_{\overline{X}}$ is the diagonal of $\overline{X}^2$, let $\{(k_i,z_i)\}_{i \in I} \subseteq \overline{X}^2$ be such that $\{(\overline{k}_i,\overline{z}_i)\}_{i \in I} = N$, and let $K = \RS_{\Free(\overline{X})}(N)$. Let $p, q \in T(X)$ be such that $(p^{\Free(\overline{X})}(\overline{x}_1,\dots,\overline{x}_n),q^{\Free(\overline{X})}(\overline{x}_1,\dots,\overline{x}_n)) \in K$. Then there exists $t \in T(X \cup Y)$ with $Y = \{y_1,\dots,y_m\}$ and such that the equations:
	\begin{align*}
		p(x_1,...,x_n) &\approx t(x_1,...,x_n,k_1,...,k_m)
		\\q(x_1,...,x_n) &\approx t(x_1,...,x_n,z_1,...,z_m)
	\end{align*}
	hold in $\vv{V}$.
\end{Lem}

With these tools we are ready to show an analogon of  the Pixley-Wille algorithm \cite{Pix.LMC}, \cite{Wil.K} for equations concerning compatible reflexive relations, which was claimed in \cite[pages 273-274]{Tsc.MCET} without a proof. Thus we show a proof of the algorithm in \cite[pages 273-274]{Tsc.MCET} using a similar technique to \cite[Proposition $1.4$]{KK.TTPI}. 

\begin{Prop}
	\label{Propmain}
	Let $p$ and $q$ be two $n$-ary terms for the language $\{\wedge, \circ, +\}$ with $p$ $+$-free. The class of varieties such that $p \leq q$ holds in $\Crr(A)$ for all $A \in \vv{V}$ is a Mal’cev class described by the family of equations $\{\Eq^R(p \leq q^{(k)}) \}_{k \in \NN\backslash\{1\}}$. Furthermore, if $q$ is $+$-free is a strong Mal'cev class described by $\Eq^R(p \leq q)$.
\end{Prop}

\begin{proof}
	Suppose that $p \leq q$ holds for the set of all the compatible reflexive relations of algebras in $\vv{V}$, with $p$ and $q$ $n$-ary terms. Let $T_i(p)$ be constructed from $p$ as in \eqref{eqp} and let $V = \{y_1,\dots,y_m\}$ be the set of vertices of $\mathbf{G}(p)$, the graph associated with $p$. Let $\mathbf{F} = \Free(\{x_1,\dots,x_m\})$ be the  $m$-generated free algebra in $\vv{V}$. Then for all $i = 1, \dots,n$  let $R'_i$ be the compatible reflexive relation of $\mathbf{F}$ defined by:
	\begin{equation*}
		R'_i = \RS_{\mathbf{F}}(T_i(p)).
	\end{equation*}
	Then, from Proposition \ref{PropKK} item $(1)$, it follows that:
	\begin{equation*}
		(x_1,x_2) \in p(R'_1,\dots,R'_n) \subseteq q(R'_1,\dots,R'_n),
	\end{equation*}  
through the assignment $y_s \mapsto x_s$. Hence there exists a $k \in \NN\backslash\{1\}$ such that $(x_1,x_2) \in q^{(k)}(R'_1,\dots,R'_m)$. We prove that $\vv{V}$ satisfies $\Eq^{R}(p \leq q^{(k)})$. To this end let $Z = \{z_1,\dots,z_u\}$ be the set of vertices of $\mathbf{G}(q^{(k)})$. Let us consider the pair of equations
	\begin{equation} 
	\label{GoalThAlg}
	\begin{split}
	t_{(i,j,s)}(x_1,\dots,x_m,y_{i_1},\dots,y_{i_{c(s)}}) \approx \ &t_i(x_1,\dots,x_m)
	\\ t_{(i,j,s)}(x_1,\dots,x_m,y_{j_1},\dots,y_{j_{c(s)}})\approx\ &t_j(x_1,\dots,x_m),
	\end{split}
	\end{equation}
	such that $(t_i,t_j) \in Tt_s(q^{(k)})$, $t_{(i,j,s)}$ is an $(m + c(s))$-ary term, where $c(s) = |T_s(p)|$, and the set of pairs $\{(y_{i_t},y_{j_t})\}_{1 \leq t \leq c(s)} = T_s(p)$. We know that $(x_1 ,x_2) = (\pi_1^m,\pi_2^m) \in q^{(k)}(R'_1,\dots,R'_n)$ thus, by $(2)$ of Proposition \ref{PropKK}, there is an assignment $z_s \mapsto r_s(x_1,\dots,x_m)$ from $Z \rightarrow \mathbf{F}$ that extends $z_1 \mapsto x_1$, $z_2 \mapsto x_2$ such that $(r_i,r_j) \in R'_s$ whenever $(z_i, z_j)$ is an edge labelled with $X_s$ of $\mathbf{G}(q^{(k)})$. Hence $(r_i,r_j) \in R'_s$ implies, by Lemma \ref{Lem1-4}, that there exists an $(m + c(s))$-ary term $t_{(i,j,s)}$ such that $\vv{V}$ satisfies (\ref{GoalThAlg}), where $t_i^{\mathbf{F }}(x_1,\dots,x_m)  = r_i$ and $t_j^{\mathbf{F }}(x_1,\dots,x_m) = r_j$.
	
	Conversely, suppose that there exists $k \in \NN\backslash\{1\}$ such that $\Eq^R(p \leq q^{(k)})$ holds in $\vv{V}$. Then let $\mathbf{A} \in \vv{V}$, $a_1,a_2 \in A$, and let $R'_1,\dots,R'_n \in \RS(\mathbf{A})$ be such that:
	\begin{equation*}
		(a_1,a_2) \in p(R'_1,\dots,R'_n)
	\end{equation*}	
	we want to prove that 
	\begin{equation}
	\label{GoalnewPW}
	(a_1, a_2) \in q^{(k)}(R'_1,\dots,R'_n).
	\end{equation}
	From $(2)$ of Proposition \ref{PropKK} we have that there exists an assignment $\psi:Y \rightarrow A$ extending $y_1 \mapsto a_1$,$y_2 \mapsto a_2$ such that $y_s \mapsto a_s$, with $(a_i,a_j) \in R'_s$ whenever $(y_i,y_j)$ is an $X_s$-labelled edge of $\mathbf{G}(p)$. 
	
	Let $Z = \{z_1,\dots,z_u\}$ be the set of vertices of $\mathbf{G}(q^{(k)})$ and let $|T_s(p)| = c(s) $. From the hypothesis we have that for all $(t_i,t_j) \in T_s(q^{(k)})$, there exists an $m + c(s)$-ary term $t_{(t_i,t_j,s) }$ such that:
	\begin{align*}
		t_{(i,j,s)}(x_1,\dots,x_m,y_{i_1},\dots,y_{i_{c(s)}}) &\approx t_i(x_1,\dots,x_m)
		\\ t_{(i,j,s)}(x_1,\dots,x_m,y_{j_1},\dots,y_{j_{c(s)}}) &\approx t_j(x_1,\dots,x_m)
	\end{align*}     
	are satisfied in $\mathbf{A}$, where $\{(y_{i_t},y_{j_t})\}_{1 \leq t \leq c(s)} $ $= T_s(p)$. From the definition of the assignment $\psi$ we have:
	\begin{align*}
		t_i(a_1,\dots,a_m) &= t_{(i,j,s)}(a_1,\dots,a_m,a_{i_1},\dots,a_{i_{c(s)}}) R'_s \\&t_{(i,j,s)}(a_1,\dots,a_m,a_{j_1},\dots,a_{j_{c(s)}}) = \\&=t_j(a_1,\dots,a_m).
	\end{align*}
	Let $\rho:Z \rightarrow A$ be the assignment such that $z_1 \mapsto a_1$,$z_2 \mapsto a_2$ and $z_i \mapsto t_i(a_1,\dots,a_m)$ for all $3 \leq i \leq u$. Thus we have that $(t_i(a_1,\dots,a_m), $ $ t_j(a_1,\dots,a_m)) \in R'_s$ whenever $(z_i,z_j) \in \mathbf{G}(q^{(k)})$. By $(1)$ of Proposition \ref{PropKK}, we have that $(a_1,a_2) \in q^{(k)}(R'_1,\dots,R'_n) \subseteq q(R'_1,\dots,R'_n)$. Thus the class of varieties such that $p \leq q$ holds in $\Crr(A)$ for all $A \in \vv{V}$ is a Mal’cev class described by the family of equations $\{\Eq^R(p \leq q^{(k)}) \}_{k \in \NN\backslash\{1\}}$. Furthermore, if $q$ is $+$-free is a strong Mal'cev class described by $\Eq^R(p \leq q)$.
\end{proof}

\begin{proof}[Proof of Theorem \ref{sunewPWThm}]
The proof follows from Proposition \ref{Propmain}.
\end{proof}

We show an application of this Algorithm.
\begin{Thm}
	
	Let $\vv{V}$ be a variety and let $n \in \mathbb{N}\backslash\{1\}$. Then the following are equivalent:
	
	\begin{enumerate}
		
		\item[(1)] $\vv{V}$ satisfies the compatible reflexive relation inequality: 
		\begin{equation*}
			R \wedge (S \circ T) \leq  T \circ S \circ (R \wedge S) \circ^{(n)} (R \wedge T);
		\end{equation*}

		\item[(2)] there exist $2n+2$ $4$-ary terms $s_0,s_1,t_0,...,t_{2n-1}$ and $n+3$ ternary terms $p_0,p,m_0,...,$ $m_n$ such that $\vv{V}$ satisfies:
		\begin{align*}
			&x \approx p_0(x,y,z), m_n(x,y,z) \approx z&
			\\&x \approx s_0(x,y,z,y)&
			\\&p(x,y,z) \approx s_0(x,y,z,z) \approx s_1(x,y,z,x)&
			\\&m_0(x,y,z) \approx s_1(x,y,z,y)&
			\\&t_{2i}(x,y,z,x)\approx m_i(x,y,z) &
			\\&t_{2i}(x,y,z,z)\approx m_{i+i}(x,y,z) &
			\\&t_{2i+1}(x,y,z,x)\approx m_i(x,y,z) &\text{for $i$ even}
			\\&t_{2i+1}(x,y,z,y)\approx m_{i+i}(x,y,z) &\text{for $i$ even}
			\\&t_{2i+1}(x,y,z,y)\approx m_i(x,y,z)  &\text{for $i$ odd}
			\\&t_{2i+1}(x,y,z,z)\approx m_{i+i}(x,y,z)  &\text{for $i$ odd.}
		\end{align*}
	\end{enumerate}
	
\end{Thm}

\section{Equations equivalent over congruences and over compatible reflexive relations}
\label{SecLst}

In this section we give a characterization of an interesting set of equations which describe the same Mal'cev condition when considered with variables that run over the set of all compatible reflexive relations or over the congruence lattices. . In \cite{Lip.FCIT} P. Lipparini showed that if $p$ is a regular term, a variety satisfies the congruence identity $p(\alpha_1,\dots, \alpha_n) \leq q(\alpha_1,\dots,\alpha_n)$ if and only if satisfies the tolerance identity $p(T_1,\dots,$ $T_n) \leq q(T_1,\dots,$ $T_n)$, provided we restrict the set of the variables to run over the representable tolerances , i.e. tolerances of the form $R \circ R^{-1}$ where $R$ is a compatible reflexive relation. It is also proved that both the properties to be regular for a term and to be representable for tolerances are not so restrictive. Instead, in the case of the compatible reflexive relations, the hypotheses on the equation $p \leq q$ are restrictive.

\begin{Thm}
	\label{Theq}
	Let $p,q$ be terms of the same arity for the language $\{\circ, \wedge\}$ with $p$ regular and $q$ with at most one occurrence of $\circ$. Let $\vv{V}$ be a variety. Then the following are equivalent:	
	\begin{enumerate}
		\item[(1)] $\vv{V}$ satisfies the compatible reflexive relation inequality $p \leq q$; 
		\item[(2)]$\vv{V}$ satisfies the congruence inequality $p \leq q$. 
	\end{enumerate}
\end{Thm}

\begin{proof}
	The implication $(1)\Rightarrow(2)$ is obvious. For $(2)\Rightarrow(1)$ let $\vv{V}$ be a variety that satisfies $p \leq q$. Let $\Eq(p \leq q)$ the set of equation obtained using the Algorithm \ref{PixWilAlg}. From \cite{Pix.LMC}, \cite{Wil.K} we have that $\vv{V}$ satisfies the Mal’cev-condition induced by $\Eq(p \leq q)$. Let $\mathbf{A}$ be an algebra, $R_1^*,\dots,R_n^* \in \Crr(\mathbf{A})$, and let $a_1,a_2 \in A$ be such that $(a_1,a_2) \in p(R_1^*,\dots,R_n^*)$. By $(2)$ of Proposition \ref{PropKK}, there exist elements $a_3,\dots,a_m$ $\in A$, where $m$ is the number of vertices $Y = \{y_1,\dots,y_m\}$ of $\mathbf{G}(p)$, and an assignment $\psi: Y \rightarrow A$ such that $y_i \mapsto a_i$, for $1 \leq i \leq m$, and $(a_i,a_j) \in R^*_k$ whenever $(y_i,y_j)$ is an edge of $\mathbf{G}(p)$ with label $X_k$. Let us denote by $\mathbfsl{a} = (a_1,\dots,a_m)$ and by $\mathbfsl{x} = (x_1,\dots,x_m)$. The goal is to prove that $(a_1,a_2) \in q(R_1^*,\dots,R_n^*)$. We need to distinguish between two cases. 
	\newline Case $q$ with no occurrence of $\circ$: 
	
	In this case we obtain equations satisfied by either only  the trivial variety or all varieties. We have that $\mathbf{G}(q)$ has only two vertices $\{z_1,z_2\} =Z$. Let us fix the assignment $\rho: Z \rightarrow A$ such that $z_1 \mapsto a_1$ and $z_2 \mapsto a_2$. We prove that this assignment satisfies $(1)$ of Proposition \ref{PropKK}. To this end let $(z_1,z_2)$ be an edge of $\mathbf{G}(q)$ labelled by $X_k$. By Lemma \ref{Lemeqlr}, we have that $\mathbf{A}$ satisfies the equation $\pi_1^m(\mathbfsl{x}^{(T_k(p),l)}) \approx \pi_2^m(\mathbfsl{x}^{(T_k(p),l)})$. Hence
	\begin{equation*}
		a_1 =   \pi_1^m(\mathbfsl{a}^{(T_k(p),l)}) = \pi_2^m(\mathbfsl{a}^{(T_k(p),l)})\ R^*_k\ \pi_2^m(\mathbfsl{a}) = a_2.
	\end{equation*}
	Case $q$ with exactly one occurrence of $\circ$: 
	
	In this case we apply the original Pixley-Wille algorithm \ref{PixWilAlg} to $p \leq q$ finding the set of equations $\Eq(p \leq q)$. We see that $\mathbf{G}(q)$ has only three vertices and hence the equations in $\Eq(p \leq q)$ involve only three $m$-ary terms $\pi_1^m,\pi_2^m,t$. 
		
	Let $Z= \{z_1,z_2,z_3\}$ be the set of vertices of $\mathbf{G}(q)$. Let us define the assignment $\rho: Z \rightarrow A$ such that $z_i \mapsto a_i$, for $i = 1,2$, and $z_3 \mapsto t(a_1,\dots,a_m)$. 
	Next we prove that the assignment $\rho$ satisfies the hypothesis in $(1)$ of Proposition \ref{PropKK}. To this end let us suppose that $(z_i,z_j)$ is an edge of $\mathbf{G}(q)$ labelled with $X_k$. We can observe that $(t_i,t_j) \in Tt_k(q)$ where $t_1 = \pi_1^m$, $t_2 = \pi_2^m$, and $t_3 = t$. Then we have three cases.
	
	Subcase $(i,j) = (1,2)$: in this case the proof is identical to the one of  $q$ with no occurrence of $\circ$.

	Subcase $(i,j) = (1,3)$: by Lemma \ref{Lemeqlr}, we have that $\mathbf{A}$ satisfies the equation $\pi_1^m(\mathbfsl{x}^{(T_k(p),l)}) \approx t(\mathbfsl{x}^{(T_k(p),l)})$. Hence
	\begin{equation*}
		a_1 =  \pi_1^m\ (\mathbfsl{a}^{(T_k(p),l)})\ = t(\mathbfsl{a}^{(T_k(p),l)}) \ R^*_k\ t(\mathbfsl{a}).
	\end{equation*}
	Subcase $(i,j) = (3,2)$: in this case the proof is symmetrical to Subcase $(i,j) = (1,3)$. By Lemma \ref{Lemeqlr}, we have that $\mathbf{A}$ satisfies the equation $\pi_2^m(\mathbfsl{x}^{(T_k(p),r)}) \approx t(\mathbfsl{x}^{(T_k(p),r)})$. Hence
	\begin{equation*}
		t(\mathbfsl{a})  \ R^*_k\ t(\mathbfsl{a}^{(T_k(p),r)})  = \pi_2^m(\mathbfsl{a}^{(T_k(p),r)}) = a_2.
	\end{equation*}	
	Thus $\rho$ satisfies the hypothesis in $(1)$ of Proposition \ref{PropKK} and $(a_1,a_2) \in q(R_1^*,\dots,R_n^*)$.
\end{proof}

We can see that the hypotheses of Theorem \ref{Theq} are difficult to expand. Indeed the equation $X \circ X \leq X$ for compatible reflexive relations implies congruence permutability but is satisfied by all the varieties if the variables run over the congruence lattices. Moreover, the equation $R \wedge (S \circ T) \leq (R \wedge  S) \circ (R \wedge  T) \circ (R \wedge  S)$ implies majority for compatible reflexive relations (\cite{Lip.RII3}), but when the variables run over the congruence lattices implies only $3$-distributivity, which is strictly weaker. With the following theorem we show an example of application of Theorem \ref{Theq} that can be deduced also from \cite[page $6$]{Lip.RII3}.

\begin{Thm}
	\label{ThMaj}
	Let $\vv{V}$ be a variety. Then the following are equivalent:
	
	\begin{enumerate}
		
		\item[(1)] $\vv{V}$ satisfies the congruence inequality 
		\begin{center}
			$\alpha \wedge (\beta \circ \gamma) \leq (\alpha \wedge \beta) \circ \gamma$;
		\end{center}

		\item[(2)] $\vv{V}$ satisfies the compatible reflexive relation inequality 
		\begin{center}
			$R \wedge (S \circ T) \leq (R \wedge S) \circ T$;
		\end{center}
		
	\end{enumerate}
	
\end{Thm}

The proof of this theorem follows from Theorem \ref{Theq} and clearly these two conditions imply majority. 

We can also observe the following equivalence of equations.

\begin{Thm}
\label{ThR}
	Let $\mathbf{A}$ be an algebra. Then the following are equivalent:
	
	\begin{enumerate}
	
	\item[(1)] $\mathbf{A}$ satisfies the compatible reflexive relation inequality 
	\begin{equation*}
		R \wedge (S \circ T) \leq (R \wedge S) \circ T
	\end{equation*}

	\item[(2)] $\mathbf{A}$ satisfies the compatible reflexive relation inequality 
	\begin{equation*}
		R \wedge (S \circ T) \leq (R \wedge S) \circ (R \wedge T).
	\end{equation*}
	
\end{enumerate}	
\end{Thm}

\section*{Acknowledgements}

The author thanks Paolo Alglianò, Erhard Aichinger, Sebastian Kreinecker, and Bernardo Rossi for many hours of fruitful discussions. The author thanks the referees for their useful suggestions and mention that Theorem \ref{ThR} was suggested by one of the referees.

\bibliographystyle{alpha}

\end{document}